\documentclass[12pt]{article}
\usepackage{epsfig,amsmath,amssymb,amsthm,latexsym}
\usepackage{tikz}
\usepackage{color}
\usetikzlibrary{angles,quotes} 
\usetikzlibrary{calc} 

\newcommand{\RR}{\mathbb{R}}

\def\bfa{\mathbf{a}}
\def\bfb{\mathbf{b}}

\def\bfe{\mathbf{e}}
\def\bff{\mathbf{f}}
\def\bfg{\mathbf{g}}

\def\bfp{\mathbf{p}}
\def\bfq{\mathbf{q}}
\def\bfr{\mathbf{r}}

\def\sgn{{\rm sgn}}

\def\bfv{\mathbf{v}}
\def\bfx{\mathbf{x}}
\def\bfy{\mathbf{y}}
\def\Pbar{{\overline P}}
\def\Qbar{{\overline Q}}

\def\bfgamma{{\boldsymbol \gamma}}

\newtheorem{theorem}{Theorem}
\newtheorem{lemma}{Lemma}

\title{On the injectivity of mean value mappings between quadrilaterals}
\author{
Michael S. Floater\footnote{
Department of Mathematics,
University of Oslo,
PO Box 1053, Blindern,
0316 Oslo,
Norway,
{\it email: michaelf@math.uio.no}}
\and
Georg Muntingh\footnote{
SINTEF, PO Box 124 Blindern, 0314 Oslo, Norway,
{\it email: Georg.Muntingh@sintef.no}}
}                                                                               
                                                                                
\begin{document}

\maketitle

\begin{abstract}
Mean value coordinates can be used to map one polygon into another,
with application to computer graphics and curve and surface modelling.
In this paper we show that if the polygons are quadrilaterals,
and if the target quadrilateral is convex, then the mapping is injective.
\end{abstract}


\noindent {\em Math Subject Classification: 65D17, 26B10}

\section{Introduction}

Barycentric mapping has emerged as a convenient tool
for deforming shapes when modelling
and processing geometry.
For example, a planar curve can be deformed by
enclosing it in a polygon and deforming the polygon.
Using some choice of generalized barycentric coordinates
(GBCs) to represent each point of the polygon, the polygon can be deformed
by moving its vertices.
In effect, this defines a mapping $\bff: P \to Q$
from the initial polygon $P$,
the `domain' polygon, to a new polygon $Q$, the `target' polygon.
An early paper proposing such barycentric mapping both in $\RR^2$
and higher Euclidean dimension is that of
Warren~\cite{Warren:96}, who used a generalization of Wachspress'
rational coordinates for the barycentric representation of points
in a convex polygon or polytope.
Later, the discovery of mean value coordinates
enhanced the practicality of barycentric mapping
since these coordinates allow the domain to be
non-convex~\cite{Floater:03a,Dyken:09,Ju:05a,Floater:05,
Hormann:06,Lipman:07,Weber:11}.

In many applications of curve deformation we would
like to avoid the deformed curve having self-intersections.
Thus, we would like some kind of condition on $P$ and $Q$
that guarantees that $\bff$ is injective.
This was the motivation for the theoretical study made in~\cite{Floater:10a},
the main result of which was to show that
if $P$ and $Q$ are both convex polygons, then
the mapping determined by Wachspress's coordinates is injective.
It was also shown there, by way of counterexamples, that,
with $P$ and $Q$ again convex,
the mean value mapping is not always injective
when the number of vertices in $P$ (and in $Q$) is five or more.
This drawback of mean value coordinates was the motivation for
the method proposed in~\cite{Schneider:13},
in which an injective mapping is constructed
as the functional composition of several mean value mappings,
each of which is an injective perturbation of the identity mapping.

The results of~\cite{Floater:10a} left open the important case:
is a mean value mapping injective when $P$ and $Q$ are convex
quadrilaterals?
The purpose of this paper is to settle this question.
We will show in fact a more general result.
\begin{theorem}\label{thm:main}
If $P$ is any quadrilateral, convex or non-convex, and
if $Q$ is convex, then the mean value mapping $\bff:P\to Q$ is
injective.
\end{theorem}

Figures~\ref{fig:mvconvex} and~\ref{fig:mvconcave} illustrate the theorem.
\begin{figure}[ht]
\centerline{\includegraphics[width=0.6\textwidth]{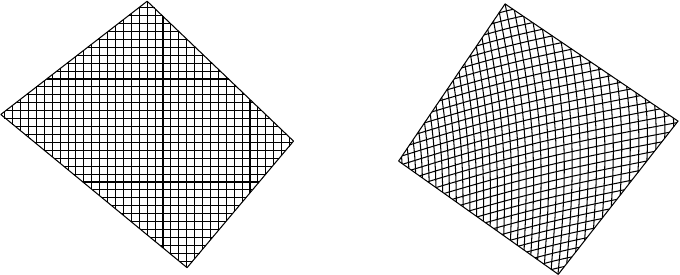}}
\caption{MV mapping, convex to convex.} 
\label{fig:mvconvex}
\end{figure}
\begin{figure}[ht]
\centerline{\includegraphics[width=0.6\textwidth]{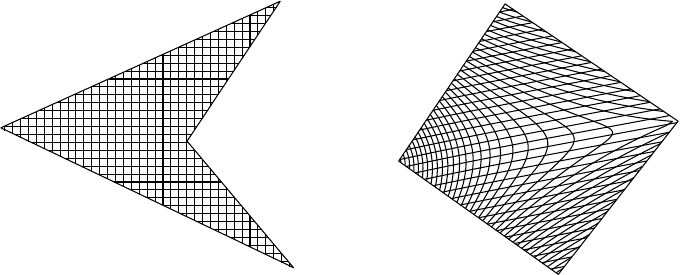}}
\caption{MV mapping, concave to convex.} 
\label{fig:mvconcave}
\end{figure}

So in the quadrilateral case, mean value mappings have the same
injectivity property as harmonic mappings;
see~\cite{Rado:26,Kneser:26,Choquet:45}.

The essential part of the proof of Theorem~\ref{thm:main} is to show that
when the vertices of~$P$ and~$Q$ have the same orientation,
the Jacobian (determinant) $J(\bff)$ is positive both in~$P$
and at all boundary points of $P$ except the four vertices.
This requires computing the gradients of
the mean value coordinates.
We will derive in Theorem~\ref{thm:omegaigradient} a new, simple
formula for the gradients of the homogeneous form of the coordinates.
This formula applies to an arbitrary polygon and might be
useful in other applications, for example, to derive bounds on
the gradients for finite element applications;
see~\cite{Rand:13,Floater:14a,Floater:16}.

\section{Barycentric mappings}\label{sec:mappings}

Let $P \subset \RR^2$ be a polygon, viewed as an open set, with vertices
$\bfp_1,\ldots,\bfp_n$, ordered anticlockwise.
We denote by $\partial P$ the boundary of $P$
and by $\Pbar$ its closure.
Let $\phi_1,\ldots,\phi_n: \Pbar \to \RR$
be continuous functions such that
\begin{equation} \label{eq:bary1}
 \sum_{i=1}^n \phi_i(\bfx) = 1, \quad \bfx \in \Pbar,
\end{equation}
\begin{equation} \label{eq:bary2}
 \sum_{i=1}^n \phi_i(\bfx) \bfp_i = \bfx, \quad \bfx \in \Pbar,
\end{equation}
and such that if
\begin{equation} \label{eq:xedge}
\bfx = (1-\mu)\bfp_\ell + \mu \bfp_{\ell+1}
\end{equation}
for some $\ell$ and some $\mu \in [0,1]$, with indices treated cyclically
($\bfp_{n+1} := \bfp_1$ and so on), then
\begin{equation}\label{eq:phibdy}
 \phi_\ell(\bfx) = 1-\mu, \quad \phi_{\ell+1}(\bfx) = \mu,
  \quad \hbox{and} 
\quad \hbox{$\phi_i(\bfx) = 0, \quad i \ne \ell,\ell+1$}.
\end{equation}
Then $\phi_1,\ldots,\phi_n$ are
\emph{generalized barycentric coordinates} (GBCs) for $P$.

Next suppose $Q \subset \RR^2$ is another polygon, with vertices
$\bfq_1,\ldots,\bfq_n$, ordered anticlockwise.
Then we call the mapping $\bff : \Pbar \to \RR^2$ defined by
\begin{equation} \label{eq:bary3}
 \bff(\bfx) = \sum_{i=1}^n
      \phi_i(\bfx) \bfq_i, \qquad \bfx \in \Pbar,
\end{equation}
a \emph{barycentric mapping}.

Equation (\ref{eq:bary1}) implies that for $\bfx \in \Pbar$,
the point $\bff(\bfx)$ is an affine combination of the
points $\bfq_i$. 
If the $\phi_i$ are non-negative then this combination is also convex
and if in addition~$Q$ is convex then $\bff(\Pbar) \subseteq \Qbar$.
Moreover, if the $\phi_i$ are positive in $P$, then
$\bff(P) \subseteq Q$.

The linear precision property (\ref{eq:bary2}) implies that
if  $Q = P$ then $\bff$ is the identity mapping,
and by the continuity of the $\phi_i$, $\bff$ is continuous
and so if~$Q$ is a perturbation of $P$ then $\bff$
is a perturbation of the identity mapping.
By~(\ref{eq:phibdy}),
if $\bfx$ is the boundary point in (\ref{eq:xedge}) then
\begin{equation} \label{eq:fedge}
 \bff(\bfx) = (1-\mu)\bfq_\ell + \mu \bfq_{\ell+1}.
\end{equation}
Thus $\bff$ maps $\partial P$ to $\partial Q$
in a piecewise linear fashion,
mapping vertices and edges of $\partial P$
to corresponding vertices and edges of $\partial Q$.

The question we will address is whether $\bff$ is injective.
If both $P$ and $Q$ are convex and if the GBCs are Wachspress coordinates
then $\bff$ is $C^\infty$ on~$\Pbar$ (technically, $\bff$
is $C^\infty$ in an open set containing $\Pbar$), and
it was shown in~\cite{Floater:10a} that
the Jacobian of $\bff$ is positive on $\Pbar$.
Then, by a theorem of \cite{Kestelman:1971},
the positivity of $J(\bff)$ on $\Pbar$ implies that $\bff$ is injective.

We will use a similar approach to prove~Theorem~\ref{thm:main}.
But first let us observe that there is
a simple formula for $J(\bff)$ at a vertex of $P$
for \emph{any} barycentric mapping $\bff$
\emph{as long as} $\bff$ is $C^1$ at the vertex.
Recall that the Jacobian of~$\bff$ is the determinant
$$ J(\bff) := | \bff_x, \bff_y| =
   \left| \begin{matrix}
     f_x & f_y \\
     g_x & g_y
   \end{matrix}
   \right|, $$
where
$$ h_x := \frac{\partial h}{\partial x}, \qquad
   h_y := \frac{\partial h}{\partial y}, $$
and $\bfx = (x,y)$ and $\bff(\bfx) = (f(\bfx),g(\bfx))$.
For points $\bfp = (p^1,p^2)$, $\bfq = (q^1,q^2)$, $\bfr = (r^1,r^2)$
in $\RR^2$ let
\begin{equation}\label{eq:triangdet}
   A(\bfp,\bfq,\bfr) :=
\frac{1}{2} \left| \begin{matrix}
     1   & 1   & 1 \\
     p^1 & q^1 & r^1 \\
     p^2 & q^2 & r^2 \\
     \end{matrix}
   \right|,
\end{equation}
the signed area of the triangle
$[\bfp,\bfq,\bfr]$.

We will often use the two-dimensional cross product,
$$ \bfa \times \bfb :=
\left| \begin{matrix}
	a^1 & b^1 \\ a^2 & b^2
 \end{matrix}
 \right| = a^1b^2 - a^2b^1 $$
for vectors $\bfa = (a^1,a^2)$ and $\bfb = (b^1,b^2)$.
So, for example,
$$ J(\bff) = \bff_x \times \bff_y, \qquad
   A(\bfp,\bfq,\bfr) =
  \frac{1}{2} (\bfq-\bfp) \times (\bfr - \bfq). $$

\begin{lemma}\label{lem:Jvertex}
If $\bff:\Pbar \to \RR^2$ is any barycentric mapping
that is $C^1$ at $\bfp_i$ then
$$ J(\bff)(\bfp_i) =
	\frac{A(\bfq_{i-1},\bfq_i,\bfq_{i+1})}
	{A(\bfp_{i-1},\bfp_i,\bfp_{i+1})}.
$$
\end{lemma}

\begin{proof}
At a point where $\bff$ is $C^1$,
it has a directional derivative $D_\bfv \bff$ in
any vector direction $\bfv=(v^1,v^2)$ given by
$$ D_\bfv \bff := 
\lim_{\epsilon \to 0} \frac{\bff(\bfx+ \epsilon \bfv) -\bff(\bfx)}{\epsilon}
= v^1 \bff_x + v^2 \bff_y. $$
Then at $\bfp_i$,
\begin{align}
\label{eq:direcderv1}
D_{\bfp_i - \bfp_{i-1}} \bff &= 
(\bfp_i^1 - \bfp_{i-1}^1) \bff_x + 
(\bfp_i^2 - \bfp_{i-1}^2) \bff_y, \\
\label{eq:direcderv2}
D_{\bfp_{i+1} - \bfp_i} \bff &= 
(\bfp_{i+1}^1 - \bfp_i^1) \bff_x + 
(\bfp_{i+1}^2 - \bfp_i^2) \bff_y,
\end{align}
and by the piecewise linearity of $\bff$ on $\partial P$ given by
(\ref{eq:xedge}) and (\ref{eq:fedge}),
\begin{align*}
 D_{\bfp_{i+1}-\bfp_i} \bff
&= \lim_{\mu \to 0+} \frac{\bff(\bfp_i + \mu (\bfp_{i+1}-\bfp_i)) -\bff(\bfp_i)}{\mu} \cr
&= \lim_{\mu \to 0+} \frac{(\bfq_i + \mu (\bfq_{i+1}-\bfq_i)) -\bfq_i}{\mu}
= \bfq_{i+1} -\bfq_i,
\end{align*}
and similarly,
$$ D_{\bfp_i-\bfp_{i-1}} \bff = \bfq_i -\bfq_{i-1}. $$
So from (\ref{eq:direcderv1}--\ref{eq:direcderv2}) we deduce that
\begin{align*}
	& (\bfq_i -\bfq_{i-1} ) \times (\bfq_{i+1}-\bfq_i) =
(D_{\bfp_i-\bfp_{i-1}}\bff)  \times (D_{\bfp_{i+1}-\bfp_i} \bff) \cr
	&= \big((\bfp_i^1 - \bfp_{i-1}^1) (\bfp_{i+1}^2 - \bfp_i^2)
- (\bfp_i^2 - \bfp_{i-1}^2) (\bfp_{i+1}^1 - \bfp_i^1) \big)
	(\bff_x  \times \bff_y),
\end{align*}
and so
$$ J(\bff) =
\frac{(\bfq_i -\bfq_{i-1} ) \times (\bfq_{i+1}-\bfq_i)}
{(\bfp_i -\bfp_{i-1} ) \times (\bfp_{i+1}-\bfp_i)}.
$$
\end{proof}

Suppose now that $\bff:\Pbar \to \RR^2$ is a barycentric mapping
that is $C^1$ on $\Pbar$.
If both $P$ and $Q$ are convex then at each vertex $\bfp_i$,
both $A(\bfp_{i-1},\bfp_i,\bfp_{i+1})$ and
$A(\bfq_{i-1},\bfq_i,\bfq_{i+1})$ are positive
and the lemma implies that $J(\bff)$ is positive at~$\bfp_i$.
This agrees with the Wachspress case studied in~\cite{Floater:10a}.
On the other hand, suppose that $P$ is non-convex and $Q$ convex.
Then $P$ must have a non-convex
vertex, i.e. a vertex $\bfp_i$ with $A(\bfp_{i-1},\bfp_i,\bfp_{i+1}) < 0$.
Then $J(\bff)$ is negative at $\bfp_i$ while it is positive
at the convex vertices of $P$.
We conclude that $J(\bff)$ must change sign in $\partial P$
and $\bff$ cannot be injective.
The same conclusion can be drawn in the case that $P$ is convex
and $Q$ is concave.
A similar observation concerning polynomial mappings from the unit square
to a concave quadrilateral was made in \cite{Gravesen:2014}.

With these facts in mind, it may seem surprising that
the mean value mapping from a concave quadrilateral to
a convex quadrilateral can be injective.
But this does not contradict Lemma~\ref{lem:Jvertex} because
mean value coordinates are \emph{not}~$C^1$ at the polygon vertices.
This lack of smoothness can thus be viewed as
an \emph{advantage} of mean value coordinates, compared
with, for example, Wachspress coordinates.
This same lack of smoothness must also be inherent in harmonic coordinates.
A referee pointed out to us that a simple argument
in \cite[Figure 4]{Anisimov:2016} shows that if $\bfp_i$ is
a non-convex vertex and the coordinates
are non-negative in $P$ (which is the case
for harmonic coordinates) then $\phi_i$ cannot be $C^1$ at $\bfp_i$.

Suppose now that $\bff:\Pbar \to \RR^2$ is a barycentric mapping
and that $\bfx \in \Pbar$ is a point at which $f$ is $C^1$.
A useful identity was derived in~\cite{Floater:10a} that
relates $J(\bff)(\bfx)$ to the signed areas of triangles formed by
the vertices $\bfq_i$ of~$Q$.
For differentiable bivariate functions $a,b,c$, let
$D(a,b,c)$ denote the $3 \times 3$ determinant
\begin{equation}\label{eq:det}
  D(a,b,c) :=
  \left| \begin{matrix}
     a & b & c \\
     a_x & b_x & c_x \\
     a_y & b_y & c_y
     \end{matrix}
   \right|.
\end{equation}
Then \cite[Lemma 1]{Floater:10a},
\begin{equation}\label{eq:J}
 J(\bff)(\bfx) = 2 \sum_{1 \le i < j < k \le n}
                  D(\phi_i,\phi_j,\phi_k)(\bfx)
		 \, A(\bfq_i,\bfq_j,\bfq_k).
\end{equation}
From this it follows that if $Q$ is convex and if
\begin{equation}\label{eq:Dijkpos}
	D(\phi_i,\phi_j,\phi_k)(\bfx) \ge 0,
\qquad 1 \le i < j < k \le n,
\end{equation}
with strict inequality for at least one choice of
$i,j,k$, then $J(\bff)(\bfx) > 0$.

\section{Mean value coordinates}
Consider now the mean value (MV) coordinates~\cite{Floater:03a,Hormann:06}.
These can be defined as follows; see Figure~\ref{fig:mvnotation}.
\begin{figure}[ht]
\centerline{\includegraphics[width=0.25\textwidth]{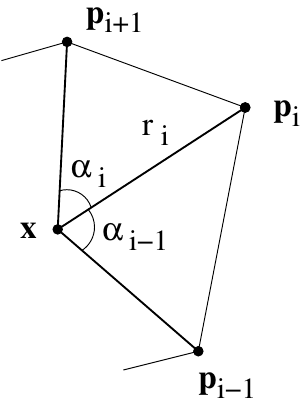}}
\caption{Notation for mean value coordinates.} 
\label{fig:mvnotation}
\end{figure}
For $\bfx \in P$,
\begin{equation}\label{eq:mv}
\phi_i(\bfx) := w_i(\bfx) / W(\bfx),
\end{equation}
where
\begin{equation}\label{eq:wi2}
w_i := \frac{t_{i-1} + t_i}{r_i}, \qquad
W := \sum_{i=1}^n w_i,
\end{equation}
and where $r_i$ is the Euclidean distance $r_i := \|\bfp_i - \bfx \|$,
and $t_i$ is the half-angle tangent $t_i := \tan(\alpha_i/2)$.
The angle $\alpha_i \in (-\pi, \pi)$ is the \emph{signed} angle at
$\bfx$ in the triangle
$[\bfx,\bfp_i,\bfp_{i+1}]$, i.e.,
$$ \sgn(\alpha_i) := \sgn(\bfe_i \times \bfe_{i+1}), $$
where $\bfe_i$ is the unit vector $\bfe_i := (\bfp_i - \bfx)/r_i$.

It was shown in~\cite{Hormann:06} that
the MV coordinates $\phi_1,\ldots,\phi_n$ are
GBCs, i.e., they satisfy the linear precision properties
(\ref{eq:bary1}) and (\ref{eq:bary2}) in $P$,
and they have a unique continuous extension to $\partial P$,
satisfying the Lagrange property (\ref{eq:phibdy}) on $\partial P$.

Due to the anticlockwise ordering of the $\bfp_i$,
$\alpha_i > 0$ in $P$ in the case that $P$ is convex,
but for general $P$,
the sign of $\alpha_i$ can be any value in $\{-1,0,1\}$.

The half-angle tangent $t_i$ comes from the integral
definition of mean value coordinates for positive $\alpha_i$
of~\cite{Floater:03a}, where it appears in the form
$$ t_i = \frac{1 - c_i}{s_i}, $$
where $s_i := \sin(\alpha_i)$ and
$c_i := \cos(\alpha_i)$.
We can therefore compute $t_i$ without needing to
evaluate trigonometric functions by the formula
\begin{equation}\label{eq:tifirst}
 t_i = \frac{1 - \bfe_i \cdot \bfe_{i+1}}{\bfe_i \times \bfe_{i+1}},
\end{equation}
when the $\alpha_i$ are positive.
This formula is also correct for $\alpha_i < 0$,
but requires defining the special case
$t_i :=0$ if $\alpha_i = 0$. This can be avoided by
using the alternative formula
\begin{equation}\label{eq:tisecond}
 t_i = \frac{s_i}{1+c_i} =
    \frac{\bfe_i \times \bfe_{i+1}}{1 + \bfe_i \cdot \bfe_{i+1}}.
\end{equation}

We note that an alternative way
of expressing $\phi_i$ even when $\bfx$ is a boundary point
was proposed recently in~\cite{Fuda:2024}, where
the issue of how to compute $\phi_i$ from the point of view of
numerical stability has been studied in depth.

\section{Gradients of mean value coordinates}

In order to apply the condition (\ref{eq:Dijkpos}) to MV coordinates,
we need a suitable formula for their gradients.
Applying the quotient rule to (\ref{eq:mv}) gives
$$ \nabla \phi_i = \frac{\nabla w_i}{W} - \frac{w_i \nabla W}{W^2}, $$
and so it is sufficient to find a formula for $\nabla w_i$.
A formula for $\nabla w_i$ was derived in~\cite[Sec.\ 4.1]{Floater:14b},
but it expresses $\nabla w_i$ as a linear combination of four vectors.
We now derive a more compact formula for $\nabla w_i$
as a linear combination of just two vectors.

Here and later we will use the well known half-angle tangent
identities
\begin{equation}\label{eq:sct}
 s_j = \frac{2t_j}{1 + t_j^2}, \qquad
   c_j = \frac{1-t_j^2}{1 + t_j^2}.
\end{equation}
For any vector $\bfa=(a_1,a_2) \in \RR^2$, we define the vector
$$ \bfa^\perp := (-a_2,a_1), $$
which is the rotation of $\bfa$ through a positive angle of $\pi/2$.
We will show:
\begin{theorem}\label{thm:omegaigradient}
\begin{equation}\label{eq:gradient}
\nabla w_i = \frac{1}{2r_i}
(u_{i-1} \bfe_{i-1}^\perp - u_i \bfe_{i+1}^\perp),
\end{equation}
where
$$ u_j := \Big( \frac{1}{r_j} + \frac{1}{r_{j+1}} \Big)
	(1 + t_j^2). $$
\end{theorem}
Notice that $u_j > 0$ and so
the theorem shows that
the vector $\nabla w_i$ is a positive linear combination of
the two vectors $\bfe_{i-1}^\perp$ and $-\bfe_{i+1}^\perp$.

\begin{proof}
Following~\cite[Sec.\ 4.1]{Floater:14b}, we
can compute $\nabla w_i$ by writing $w_i$ as
$$ w_i = \frac{t_{i-1}}{r_i} + \frac{t_i}{r_i}, $$
and applying the quotient rule to each part.
It was shown in~\cite[Sec.\ 4.1]{Floater:14b} that
$$ \nabla t_j = \frac{t_j}{s_j} 
\left(\frac{\bfe_j^\perp}{r_j} - \frac{\bfe_{j+1}^\perp}{r_{j+1}} \right), $$
which was then combined with the fact that $\nabla r_i = - \bfe_i$ to
show that
\begin{align}
\nabla \Big(\frac{t_{i-1}}{r_i} \Big)&= 
\frac{t_{i-1}}{r_i s_{i-1}}
\Big( \frac{\bfe_{i-1}^\perp}{r_{i-1}} - \frac{\bfe_i^\perp}{r_i} \Big)
+ \frac{t_{i-1}}{r_i^2} \bfe_i, \label{eq:I1grad1} \\
\nabla \Big(\frac{t_i}{r_i} \Big)&= 
\frac{t_i}{r_i s_i}
\Big( \frac{\bfe_i^\perp}{r_i} - \frac{\bfe_{i+1}^\perp}{r_{i+1}} \Big)
+ \frac{t_i}{r_i^2} \bfe_i. \label{eq:I2grad2}
\end{align}

Now let us observe that $\bfe_i$ and $\bfe_i^\perp$
form an orthonormal system of vectors in $\RR^2$, and so
$$
\bfe_{i-1}^\perp = \cos\big(\frac{\pi}{2}-\alpha_{i-1}\big) \bfe_i +
\sin\big(\frac{\pi}{2}-\alpha_{i-1}\big) \bfe_i^\perp
= s_{i-1}\bfe_i + c_{i-1}\bfe_i^\perp.
$$
Solving this for $\bfe_i$ gives
$$
\bfe_i = \frac{1}{s_{i-1}} \bfe_{i-1}^\perp
- \frac{c_{i-1}}{s_{i-1}} \bfe_i^\perp. $$
Then substituting this into (\ref{eq:I1grad1}),
we deduce
$$
\nabla \Big(\frac{t_{i-1}}{r_i} \Big)= 
\frac{t_{i-1}}{r_i s_{i-1}}
\Big( \frac{1}{r_{i-1}} + \frac{1}{r_i} \Big)
  \bfe_{i-1}^\perp
- \frac{t_{i-1}}{r_i^2 s_{i-1}}
 (1 + c_{i-1}) \bfe_i^\perp,
 $$
and using the facts
$$ \frac{t_j}{s_j} = \frac{1 + t_j^2}{2}, \qquad
	\frac{t_j(1+c_j)}{s_j} = 1, $$
we find
\begin{equation}\label{eq:I1grad1new}
\nabla \Big(\frac{t_{i-1}}{r_i} \Big)= 
\frac{u_{i-1}}{2r_i} \bfe_{i-1}^\perp
-\frac{\bfe_i^\perp}{r_i^2}. 
\end{equation}

Similarly,
$$
\bfe_{i+1}^\perp = \cos\big(\frac{\pi}{2}+\alpha_i\big) \bfe_i +
\sin\big(\frac{\pi}{2}+\alpha_i\big) \bfe_i^\perp
= -s_i\bfe_i + c_i\bfe_i^\perp,
$$
and so
$$
\bfe_i = \frac{c_i}{s_i} \bfe_i^\perp 
- \frac{1}{s_i} \bfe_{i+1}^\perp, $$
and substituting this into
(\ref{eq:I2grad2}) leads to
$$
\nabla \Big(\frac{t_i}{r_i} \Big)= 
\frac{t_i}{r_i^2 s_i} (1 + c_i) \bfe_i^\perp
-\frac{t_i}{r_i s_i}
\Big( \frac{1}{r_i} + \frac{1}{r_{i+1}} \Big)
  \bfe_{i+1}^\perp,
 $$
and therefore,
\begin{equation}\label{eq:I1grad2new}
\nabla \Big(\frac{t_i}{r_i} \Big) = 
\frac{\bfe_i^\perp}{r_i^2} 
- \frac{u_i}{2r_i} \bfe_{i+1}^\perp.
\end{equation}
Taking the sum of (\ref{eq:I1grad1new}) and 
(\ref{eq:I1grad2new}),
the term $\bfe_i^\perp/r_i^2$ cancels out
and we are left with the formula (\ref{eq:gradient}).
\end{proof}

We initially derived (\ref{eq:gradient}) using the integral formula
for $w_i$ and differentiating the integrand.
However, this method of proof is rather long, and we saw,
with the benefit of hindsight,
that (\ref{eq:gradient}) follows more easily from the calculations
of~\cite{Floater:14b}.

\section{Quadrilateral MV mappings}

We will assume \emph{for the rest of the paper} that
$P$ is a quadrilateral, i.e., $n=4$, and
that $\phi_1,\phi_2,\phi_3,\phi_4$ are
the MV coordinates with respect to $P$.
Then the determinant condition (\ref{eq:Dijkpos}) reduces to
\begin{equation}\label{eq:Dijkposquad}
D_i(\bfx) := D(\phi_{i-1},\phi_i,\phi_{i+1})(\bfx) \ge 0,
\qquad i=1,2,3,4,
\end{equation}
with strict inequality for at least one $i$.
In this section we show that this condition holds for all $\bfx \in P$.

We start with an observation about the signs of the $t_i$.
By the definition of the $\alpha_i$, $\cos(\alpha_i/2) > 0$
while $\sin(\alpha_i/2)$ and $t_i$ have the same sign as~$\alpha_i$.
Thus $t_i$ may be negative but
the sum $t_{i-1} + t_i$ is positive.
\begin{lemma}\label{lem:tsumpos}
$t_{i-1} + t_i > 0$ for all $i=1,2,3,4$.
\end{lemma}
From this lemma it follows that also $w_i > 0$ and $\phi_i > 0$,
a property that was observed in~\cite{Hormann:2004}.

\begin{proof}
If $P$ is convex, then $\alpha_i > 0$ for all $i$ and
$$ t_i = \frac{\sin(\alpha_i/2)}{\cos(\alpha_i/2)} > 0. $$
It remains to consider the case that $P$ is non-convex,
as in Figure~\ref{fig:pos}.
\begin{figure}[ht]
\centerline{\includegraphics[width=0.4\textwidth]{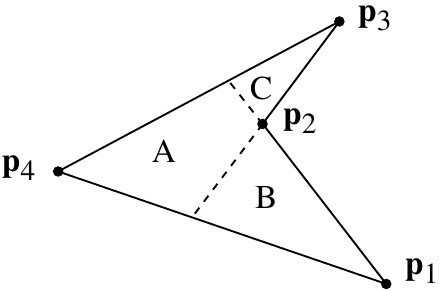}}
\caption{Non-convex quadrilateral.} 
\label{fig:pos}
\end{figure}
If $\bfx$ is in $A$ (the kernel of $P$), then $\alpha_i > 0$
for all $i$ and again $t_i > 0$. Suppose next that $\bfx \in B$.
Then $\alpha_2 < 0$ and $\alpha_1, \alpha_3, \alpha_4 > 0$,
and further $|\alpha_2| < |\alpha_1|$ and $|\alpha_2| < |\alpha_3|$.
Therefore
$$ \alpha_1 + \alpha_2 > 0, \qquad \alpha_2 + \alpha_3 > 0, $$
implying
\[ 0 < \alpha_{i-1} + \alpha_i < 2\pi, \qquad i=1,2,3,4, \]
and so
\[ \sin\left(\frac{\alpha_{i-1} + \alpha_i}{2} \right) > 0, \qquad i=1,2,3,4.\]
Hence
\[
t_{i-1} + t_i = \frac{\sin\left(\frac{\alpha_{i-1} + \alpha_i}{2} \right)}{\cos\left(\frac{\alpha_{i-1}}{2}\right)\cos\left(\frac{\alpha_i}{2} \right)} > 0.
\]
The argument for $\bfx \in C$ is similar.
\end{proof}

Another ingredient in the proof is the following trigonometric identity,
relating $t_1,t_2,t_3,t_4$.
\begin{lemma}\label{lem:Georg1}
$$ t_4 = \frac{p}{q}, \quad\hbox{where}\quad
\begin{array}{l}
p := t_1 + t_2 + t_3 - t_1t_2t_3, \\
q := t_1t_2 + t_1t_3 + t_2t_3 - 1 > 0.
\end{array}
$$
\end{lemma}
This is a special case of a more general identity for
half-angle tangents involving
symmetric polynomials~\cite{Bronstein:1989}, but
we need to take care that $q \ne 0$.
\begin{proof}
Let $s_i' := \sin(\alpha_i/2)$ and
$c_i' := \cos(\alpha_i/2)$ for $i=1,2,3,4$.
Since $-\pi < \alpha_i < \pi$, $c_i' > 0$.

Since $\sum_{i=1}^4 (\alpha_i/2) = \pi$,
applying the angle-sum formulas for sine and cosine yields
\begin{align*}
s_4' &= \sin\left(\pi-\frac{\alpha_4}{2}\right)
= \sin\left(\frac{\alpha_1}{2} + \frac{\alpha_2}{2} + \frac{\alpha_3}{2}
\right) \cr
&= c_2'c_3's_1' + c_1'c_3's_2' + c_1'c_2's_3' - s_1's_2's_3', \cr
c_4' &= -\cos\left(\pi-\frac{\alpha_4}{2}\right)
= -\cos\left(\frac{\alpha_1}{2} + \frac{\alpha_2}{2} + \frac{\alpha_3}{2}
\right) \cr
&= c_3's_1's_2' + c_2's_1's_3' + c_1's_2's_3' - c_1'c_2'c_3'.
\end{align*}
Since $t_4 = s_4'/c_4'$, we obtain $t_4 = p/q$
by dividing both $s_4'$ and $c_4'$ by $c_1'c_2'c_3'$, which is
positive. Moreover,
$$ p = \frac{s_4'}{c_1' c_2' c_3'} \quad\hbox{and}
\quad q = \frac{c_4'}{c_1' c_2' c_3'}, $$
and so $p$ has the same sign as $s_4'$ and $q$ is positive.
\end{proof}

We now derive an explicit formula for the determinants in
(\ref{eq:Dijkposquad}).
\begin{theorem}\label{thm:D123}
$$ D_i =
\frac{(t_{i-2}+t_{i-1})u_i + (t_i+t_{i+1})u_{i-1}}
{2W^2 r_{i-1} r_i r_{i+1}},
	\quad 1 \le i \le 4. $$
\end{theorem}

By Lemma~\ref{lem:tsumpos}, Theorem~\ref{thm:D123} implies that
$D_i(\bfx) > 0$ for all $\bfx \in P$ and $i=1,2,3,4$.

\begin{proof}
Observe first that by~\cite[Lemma 2]{Floater:10a},
$$ D_i = \frac{1}{W^3} D(w_{i-1},w_i,w_{i+1}), $$
and we can write this latter determinant as
$$ D(w_{i-1},w_i,w_{i+1}) =
\left| \begin{matrix}
w_{i-1} & w_i & w_{i+1} \\
\nabla w_{i-1} & 
\nabla w_i & 
\nabla w_{i+1}
 \end{matrix}
 \right|. $$
From (\ref{eq:wi2}) and Theorem~\ref{thm:omegaigradient},
$$ D(w_{i-1},w_i,w_{i+1}) = \frac{1}{4r_{i-1}r_ir_{i+1}} E_i, $$
where
\begin{equation}\label{eq:Eijk}
E_i :=
\left| \begin{matrix}
t_{i-2}+t_{i-1} & t_{i-1}+t_i & t_i+t_{i+1} \\
\bfa_{i-1} & \bfa_i & \bfa_{i+1},
 \end{matrix}
 \right|,
\end{equation}
and
$$
 \bfa_j := u_{j-1} \bfe_{j-1}^\perp - u_j \bfe_{j+1}^\perp,
$$
and it remains to show that
$$
E_i = 2W\big(
	(t_{i-2}+t_{i-1})u_i + (t_i+t_{i+1})u_{i-1}\big),
	\quad 1 \le i \le 4.
$$
We may assume without loss of generality that $i=1$ and
the task is to show that
$$ E_1 = 2W\big((t_3+t_4)u_1 + (t_1+t_2)u_4\big). $$
From (\ref{eq:Eijk}),
$$
E_1 = (t_3+t_4) D_{12} + (t_4+t_1) D_{24} + (t_1+t_2) D_{41},
$$
where $D_{i,j} := \bfa_i \times \bfa_j$, and since
$\bfa^\perp \times \bfb^\perp = \bfa \times \bfb$
for any vectors $\bfa,\bfb$ in $\RR^2$,
\begin{align*}
D_{12} &=
(u_4 \bfe_4 - u_1 \bfe_2) \times
(u_1 \bfe_1 - u_2 \bfe_3) 
= u_1 u_4 s_4 + u_2 u_4 s_3 + u_1^2 s_1 + u_1 u_2 s_2, \cr
D_{24} &=
(u_1 \bfe_1 - u_2 \bfe_3) \times
(u_3 \bfe_3 - u_4 \bfe_1)
= (u_1 u_3 - u_2 u_4) (\bfe_1 \times \bfe_3), \cr
D_{41} &=
(u_3 \bfe_3 - u_4 \bfe_1) \times
(u_4 \bfe_4 - u_1 \bfe_2)
= u_3 u_4 s_3 + u_1 u_3 s_2 + u_4^2 s_4 + u_1 u_4 s_1.
\end{align*}

Next observe that
\begin{equation}\label{eq:Wsum}
\frac{1}{2} \sum_{j=1}^4 s_j u_j
 = \sum_{j=1}^4 \Big( \frac{1}{r_j} + \frac{1}{r_{j+1}} \Big) t_j
= \sum_{j=1}^4 \frac{1}{r_j} (t_j + t_{j-1})
= \sum_{j=1}^4 w_j = W.    
\end{equation}
Therefore,
\begin{align*}
D_{12} & 
= 2W u_1 - (u_1 u_3 - u_2 u_4) s_3, \cr
D_{41} & 
= 2W u_4 + (u_1 u_3 - u_2 u_4) s_2,
\end{align*}
and so
$$ E_1 = 2W\big((t_3+t_4)u_1 + (t_1+t_2)u_4\big) + (u_1u_3 - u_2u_4) K, $$
where
$$ K := -(t_3 + t_4) s_3 + (t_4 + t_1) (\bfe_1 \times \bfe_3)
+ (t_1+t_2) s_2, $$
and it remains to show that $K = 0$.

By \eqref{eq:sct},
$$ \bfe_1 \times \bfe_3 =
\sin(\alpha_1 + \alpha_2) = s_1c_2 + c_1s_2 = 
2 \frac{(t_1+t_2)(1-t_1t_2)}{(1+t_1^2)(1+t_2^2)}, $$
and the only occurrences of $t_4$ in $K$ are in the sums
$t_3 + t_4$ and $t_4 + t_1$.
We can eliminate $t_4$ from these sums using Lemma~\ref{lem:Georg1}:
\begin{align*}
t_3 + t_4 & = \frac{t_3q + p}{q} = \frac{(1+t_3^2)(t_1+t_2)}{q}, \cr
t_4 + t_1 & = \frac{p + t_1q}{q} = \frac{(1+t_1^2)(t_2+t_3)}{q}.
\end{align*}
Thus it follows that
$$ K = \frac{2(t_1+t_2)}{(1+t_2^2)q}
\Big(- t_3(1+t_2^2) + (t_2+t_3)(1-t_1t_2) + t_2q \Big), $$
which by the definition of $q$ turns out to be zero.
\end{proof}

\section{Determinants on the boundary}

By~\cite[Remark 4.5]{Hormann:06}, $\phi_1,\phi_2,\phi_3,\phi_4$
are $C^\infty$ at all points of $\partial P$ except the vertices.
The following lemma shows that the determinant condition (\ref{eq:Dijkposquad})
holds at such points.
\begin{lemma}\label{lem:D412_boundary}
For any $\bfy$ in the relative interior of the edge $[\bfp_1,\bfp_2]$,
$$
D_i(\bfy)
=
\begin{cases} 
\frac{w_4(\bfy)} {2\|\bfp_2-\bfp_1\|} & \hbox{if $i=1$,} \cr
\frac{w_3(\bfy)} {2\|\bfp_2-\bfp_1\|} & \hbox{if $i=2$,} \cr
0 & \hbox{if $i=3,4$.}
\end{cases}
$$
\end{lemma}

\begin{proof}
We use the fact that
$$ D_i(\bfy) = \lim_{\bfx \to \bfy, \, \bfx \in P} D_i(\bfx). $$
As $\bfx \to \bfy$ for $\bfx \in P$,
$t_1 \to \infty$ while
$t_2,t_3,t_4$ and
$r_1,r_2,r_3,r_4$ have finite limits.

Consider the case $i=1$, so that for $\bfx \in P$,
$$ D_1 = \frac{(t_3+t_4)u_1 + (t_1+t_2)u_4}{2W^2 r_4 r_1 r_2}. $$
Its numerator is
$$ (t_3+t_4)\left(\frac{1}{r_1} + \frac{1}{r_2}\right)
   t_1^2 + O(t_1), \quad \hbox{as $t_1 \to \infty$}. $$
Recalling \eqref{eq:Wsum},
$$ W = \left(\frac{1}{r_1} + \frac{1}{r_2}\right) t_1
+ O(1), \quad \hbox{as $t_1 \to \infty$}, $$
and so the denominator of $D_1$ is
$$ 2r_4r_1r_2\left(\frac{1}{r_1} + \frac{1}{r_2}\right)^2 t_1^2
+ O(t_1), \quad \hbox{as $t_1 \to \infty$}. $$
Therefore,
$$ \lim_{\bfx \to \bfy, \, \bfx \in P} D_1 =
\frac{t_3+t_4} {2r_4(r_1+r_2)}
 = \frac{w_4} {2\|\bfp_2-\bfp_1\|}. $$

When $i=2$,
$$ D_2 = \frac{(t_4+t_1)u_2 + (t_2+t_3)u_1}{2W^2 r_1 r_2 r_3},
$$
and its numerator is
$$ (t_2+t_3)\left(\frac{1}{r_1} + \frac{1}{r_2}\right)
  t_1^2 + O(t_1), $$
and its denominator is
$$ 2r_1r_2r_3\left(\frac{1}{r_1} + \frac{1}{r_2}\right)^2 t_1^2
+ O(t_1), $$
and so
$$ \lim_{\bfx \to \bfy, \, \bfx \in P} D_2 =
\frac{t_2+t_3} {2r_3(r_1+r_2)}
 = \frac{w_3} {2\|\bfp_2-\bfp_1\|}. $$

For the cases $i=3,4$,
the numerator of $D_i$ is $O(t_1)$ and its denominator grows like $t_1^2$, and so
\[ \lim_{\bfx \to \bfy, \, \bfx \in P} D_i = 0. \qedhere \] 
\end{proof}

\section{Global injectivity}\label{sec:global}

We will now complete the proof of Theorem~\ref{thm:main}.
By~Theorem~\ref{thm:D123} and Lemma~\ref{lem:D412_boundary},
the Jacobian of $\bff$
is positive in $P$ and at all points of $\partial P$
except the four vertices.
So by the Inverse Function Theorem (IFT)~\cite[Theorem 9.24]{rudin:1983},
$\bff$ is locally injective at all such points.
It remains to show that this implies the
(global) injectivity of $\bff$ as a mapping from $\Pbar$ to $\Qbar$.
The method of proof uses the path-connectedness of $P$
and the compactness of~$\Pbar$,
similar to that of~\cite{Kestelman:1971}.
Let us say that $\bfx \in \Pbar$ is \emph{multiple} if there exists
$\bfy \in \Pbar$, $\bfy \ne \bfx$, such that $\bff(\bfy) = \bff(\bfx)$.
Since $\bff$ maps $\partial P$ to $\partial Q$ injectively and
since~$\bff$ maps $P$ to $Q$, $\partial P$ contains no
multiple points, and the task is to show that~$P$
contains no multiple points either.

\bigskip
\noindent {\it Proof of Theorem~\ref{thm:main}.}
Suppose, for the sake of contradiction, that
there exists some $\bfx \in P$ that is multiple.
Then choose any point $\bfx_B \in \partial P$ that is
not a vertex of $P$ and such that the interior of the
line segment $[\bfx,\bfx_B]$ lies entirely in $P$,
and let $\bfgamma(t)$ be the parameterization
$$ \bfgamma(t) := (1-t) \bfx + t \bfx_B, \quad 0 \le t \le 1. $$
Let
$$ T := \{t \in [0,1] : \hbox{$\bfgamma(t)$ is multiple} \}, $$
and let $t_* := \sup T$.

We will next show that $\bfgamma(t_*)$ is multiple.
If $t_* = 0$ then $\bfgamma(t_*) = \bfx$, which is multiple.
So suppose that $t_* > 0$.
Then there exists a sequence $(t_n)$ in~$T$ converging to $t_*$.
Since $\bfgamma$ is continuous, $\bfgamma(t_n) \to \bfgamma(t_*)$.
Since $\bfgamma(t_n)$ is multiple, there is some $\bfx_n \in \Pbar$,
$\bfx_n \ne \bfgamma(t_n)$, such that $\bff(\bfx_n) = \bff(\bfgamma(t_n))$.
Since $\Pbar$ is compact, the sequence $(\bfx_n)$
has a convergent subsequence in $\Pbar$, and so we may assume
that both $\bfgamma(t_n) \to \bfgamma(t_*)$ and
$\bfx_n \to \bfx_*$ for some $\bfx_* \in \Pbar$.

Since $J(\bff)(\bfgamma(t_*)) > 0$, by IFT
there are open sets $U$ and $V$ in $\RR^2$
such that $\bfgamma(t_*) \in U$, $\bff(\bfgamma(t_*)) \in V$,
$\bff$ is injective on $U$ and $\bff(U) = V$.
Then for large enough $n$, $\bfgamma(t_n) \in U$ and
so $\bfx_n \not\in U$ and so $\bfx_* \not \in U$, and
we deduce that $\bfx_* \ne \bfgamma(t_*)$.
On the other hand, since $\bff$ is continuous in $\Pbar$,
$\bff(\bfgamma(t_n)) \to \bff(\bfgamma(t_*))$
and $\bff(\bfx_n) \to \bff(\bfx_*)$, and therefore,
$\bff(\bfx_*) = \bff(\bfgamma(t_*))$.
Thus we have shown that $\bfgamma(t_*)$ is multiple as claimed.

We now consider the two cases: (i) $t_* = 1$ and
(ii) $t_* < 1$.

\noindent Case (i) If $t_* = 1$ then $\bfgamma(t_*) = \bfx_B$.
This implies that $\bfx_B$ is multiple which is a contradiction.

\noindent Case (ii) If $t_* < 1$ then $\bfgamma(t_*) \in P$
and so $\bfx_* \in P$ as well.
Then $J(\bff)(\bfx_*) > 0$ and so again by IFT,
there are open sets $U'$ and $V'$ in $\RR^2$
such that $\bfx_* \in U'$, $\bff(\bfx_*) \in V'$,
$\bff$ is injective on $U'$ and $\bff(U') = V'$.
Let $d := \| \bfx_* - \bfgamma(t_*) \|$.
We can find an open disk $D \subseteq U$ centred at $\bfgamma(t_*)$
with radius at most $d/2$
and an open disk $D' \subseteq U'$ centred at $\bfx_*$
with radius at most $d/2$.
Further by IFT, the inverse $\bfg$ of $\bff$ from $V$ to $U$
is continuous and the inverse $\bfg'$ of $\bff$ from $V'$ to $U'$
is continuous.
Therefore, $\bff(D) = \bfg^{-1}(D)$ and $\bff(D') = (\bfg')^{-1}(D)$
are open sets.
Let $V'' := \bff(D) \cap \bff(D')$ and
let $U_2 := \bfg(V'')$ and $U_2' := \bfg'(V'')$.
For small enough $\epsilon > 0$,
$\bfgamma(t_*+\epsilon) \in U_2$ and then
$\bff(\bfgamma(t_*+\epsilon)) \in V''$ and
$\bfx_\epsilon := \bfg'(\bff(\bfgamma(t_*+\epsilon))) \in U_2'$.
Then $\bfx_\epsilon \ne \bfgamma(t_*+\epsilon)$ and
$\bff(\bfx_\epsilon) = \bff(\bfgamma(t_*+\epsilon))$
and so $\bfgamma(t_*+\epsilon)$ is multiple which contradicts
the definition of $t_*$.
\quad $\square$

\section*{Declaration}

The authors declare that they have no conflicting interests.

\bibliography{coordinates}

\end{document}